\documentclass{article}%
\usepackage{amssymb}
\usepackage{amsfonts}
\usepackage{amsmath}
\usepackage{graphicx}%
\providecommand{\U}[1]{\protect\rule{.1in}{.1in}}
\newtheorem{theorem}{Theorem}[section]
\newtheorem{acknowledgement}[theorem]{Acknowledgement}

\newtheorem{corollary}[theorem]{Corollary}

\newtheorem{definition}[theorem]{Definition}
\newtheorem{example}[theorem]{Example}

\newtheorem{lemma}[theorem]{Lemma}

\newtheorem{proposition}[theorem]{Proposition}
\newtheorem{remark}[theorem]{Remark}

\newenvironment{proof}[1][Proof]{\noindent\textbf{#1.} }{\ \rule{0.5em}{0.5em}}
\begin{document}

\title{Limit behaviour of $\mu-$equicontinuous cellular automata }
\author{Felipe Garc\'{\i}a-Ramos\\felipegra@math.ubc.ca\\University of British Columbia}
\maketitle

\begin{abstract}
The concept of $\mu-$equicontinuity was introduced in \cite{Gilman1} to
classify cellular automata. We show that under some conditions the sequence of
Cesaro averages of a measure $\mu,$ converge under the actions of a $\mu
-$equicontinuous CA. We address questions raised in \cite{BlanchardEQ} on
whether the limit measure is either shift-ergodic, a uniform Bernoulli measure
or ergodic with respect to the CA. Many of our results hold for CA on
multidimensional subshifts.

\end{abstract}
\tableofcontents

\section{Introduction}

Cellular automata (CA) are discrete systems that depend on local rules.
Hedlund \cite{hedlund} characterized CA using dynamical properties:
$\phi:\left\{  1,2,...n\right\}  ^{\mathbb{Z}}\rightarrow\left\{
1,2,...n\right\}  ^{\mathbb{Z}}$ is a cellular automaton if and only if it is
continuous (with respect to the Cantor product topology) and shift-commuting.
This means every CA is a topological dynamical system (TDS), $i.e.$ a
continuous transformation $\phi$ on a compact metric space $X$. The dynamical
behaviour of these systems can range from very predictable to very chaotic.
Equicontinuity represents predictability. A TDS is equicontinuous if the
family $\{\phi^{i}\}$ is equicontinuous, that is, whenever two points $x,y\in
X$ are close, then $\phi^{i}(x),\phi^{i}(y)$ stay close for all $i\in
\mathbb{Z}_{+}$. Sensitivity (or sensitivity to initial conditions) is
considered a weak form of chaos. There are different classifications of
cellular automata and TDS using equicontinuity and sensitivity (see
~\cite{AkinAuslander} and \cite{KurkaEQ}).

Equicontinuity is a very strong property particularly for cellular automata;
different attempts have been made to define weaker but similar properties.
Using shift-ergodic probability measures, Gilman \cite{Gilman1}\cite{Gilman2}
introduced the concept of $\mu-$equicontinuity for cellular automata: $x\in X$
is a $\mu-$equicontinuity point if for \textit{most }$y$ close to $x$ we have
that $\phi^{i}(x),\phi^{i}(y)$ stay close for all $i\in\mathbb{Z}_{+}$. A CA
is $\mu-$equicontinuous if almost every point is $\mu-$equicontinuous. He also
introduced $\mu-$sensitivity ($\mu-$expansivity) and showed that a CA is
$\mu-$equicontinuous if and only if it is not $\mu-$sensitive.

The study of long term behaviour is a main topic of interest of dynamical
systems/ergodic theory. Long term behaviour can be studied for points, sets,
or measures. In particular one may ask if the orbit of a measure converges
weakly. Limit behaviour of measures under CA have been studied mainly for two
subclasses, linear and $\mu-$equicontinuous.

In \cite{LindXOR} Lind studied the limit behaviour of the CA on the binary
full-shift defined by adding the value of two consecutive positions
$\operatorname{mod}2$. He concluded that the weak limit of the Cesaro average
of every Bernoulli measure is the uniform Bernoulli measure. This result has
been generalized to other linear expansive CA, and it has been shown that the
Cesaro weak limit of an ergodic Markov measure is the uniform Bernoulli (or in
a more general setting the measure of maximal entropy)\cite{pivatoalgebraic}%
\cite{MassMartinezCesaro}.

In \cite{BlanchardEQ} Blanchard and Tisseur studied CA and measures tha give
equicontinuity points full measure. These CA are $\mu-$equicontinuous but
$\mu-$equicontinuous CA may not have any equicontinuity point (like in Example
\ref{product1}). They showed that the Cesaro weak limit exists and they asked
questions about the dynamical behaviour of the limit measure. In particular
they asked when the limit measure is shift-ergodic, a measure of maximal
entropy or $\phi-$ergodic. In this paper we address those questions; we show
that those three conditions are very strong.

We characterize $\mu-$equicontinuity on shifts of finite type using locally
periodic behaviour; $\mu-LEP$ (Proposition \ref{1i3}). We present a natural
generalization of Blanchard-Tisseur's result (Theorem \ref{cesaro}). In
section 3.2 we present the main results of this paper. Let $\phi$ be a CA and
$\mu$ a $\sigma-$ergodic measure that gives equicontinuity points full
measure. We show the limit measure is of maximal entropy (Theorem
\ref{mmecor}) if and only if $\phi$ is surjective and the original measure is
the measure of maximal entropy. We show that if $\phi$ is surjective then the
limit measure is shift-ergodic if and only $\mu$ is $\phi-$invariant (Theorem
\ref{sigmaergodic}). Finally we show that if the limit measure is ergodic with
respect to $\phi$ then the system is isomorphic (measurably) to a cyclic
permutation on a finite set (Corollary \ref{odometer}).

Some of our results hold for CA on multidimensional subshifts, in the cases
where they don't we present weaker analogous results. We also present several
results for $\mu-LEP$ and $\mu-$equicontinuous systems, which may not have any
equicontinuity points (like in Example \ref{product1}).

\bigskip

\begin{acknowledgement}
I would like to thank Brian Marcus and Tom Meyerovitch for their suggestions
and comments.
\end{acknowledgement}

\section{Equicontinuity and local periodicity}

\subsection{Definitions}

A \textbf{topological dynamical system (TDS)} is a pair $(X,\phi)$ where $X$
is a compact metric space and $\phi:X\rightarrow X$ is a continuous transformation.

The\textbf{\ }$n-$\textbf{window}, $W_{n}\subset$ $\mathbb{Z}^{d}$ is defined
as the cube of radius $n$ centred at the origin; a \textbf{window}%
\textit{\ }is an $n-$window for some $n$. For any set $W\subset\mathbb{Z}^{d}$
and $x\in\mathcal{A}^{\mathbb{Z}^{d}},$ $x_{W}\in$ $\mathcal{A}^{W}$ is the
restriction of $x$ to $W.$ We will endow $\mathcal{A}^{\mathbb{Z}^{d}}$ with
the Cantor (product) topology; this is the same topology obtained by the
metric given by $d(x,y)=\frac{1}{2^{m}},$ where $m$ is the largest integer
such that $x_{W_{m}}=y_{W_{m}}$. We denote the\textbf{\ balls} with
$B_{n}(x)\mathbf{:}=\left\{  z\mid d(x,z)\leq\frac{1}{2^{n}}\right\}  $.

We define the\textbf{\ full }$\mathcal{A}$\textbf{-shift} as the metric space
$\mathcal{A}^{\mathbb{Z}^{d}}.$ For $i\in\mathbb{Z}^{d}$ we will use
$\sigma_{i}:\mathcal{A}^{\mathbb{Z}^{d}}\rightarrow\mathcal{A}^{\mathbb{Z}%
^{d}}$ to denote the \textbf{shift maps} (the maps that satisfy $x_{i+j}%
=(\sigma_{i}x)_{j}$ for all $x\in X$ and $i,j\in\mathbb{Z}^{d}).$ The algebra
of sets generated by balls and their shifts is called the \textbf{algebra of}
\textbf{cylinder sets}. A subset $X\subset\mathcal{A}^{\mathbb{Z}^{d}}$ is a
\textbf{subshift} (or \textbf{shift space}) if it is closed and $\sigma_{i}%
-$invariant for all $i\in\mathbb{Z}^{d}.$ If $d=1$ we say the space is
\textbf{1D}. A \textbf{cellular automaton} \textbf{(CA)} is a pair $(X,\phi)$
where $X$ is a subshift and $\phi(\cdot):X\rightarrow X$ is a continuous
$\sigma-$commuting map, i.e. $\phi$ commutes with all the $\mathbb{Z}^{d}$
shifts. We say $(X,\phi)$ is a 1D CA if $X$ is a 1D subshift. In most of the
literature cellular automata is studied only on full-shifts. In our definition
a 1D subshift itself is a CA. Cellular automata of this kind are also known as
shift endomorphisms.

A \textbf{one sided subshift} is a set $X\subset\mathcal{A}^{\mathbb{N}}$ that
is closed and $\sigma-$invariant (i.e $\sigma(X)\subset X$)$.$

The following theorem was established in \cite{hedlund} for 1D CA on
full-shifts. The same result holds for CA on higher dimensional subshifts.

\begin{theorem}
[Curtis-Hedlund-Lyndon]Let $X$ be a shift space, and $\phi(\cdot):X\rightarrow
X$ a function . The map $\phi$ is a CA if and only if there exists a
non-negative integer $n$, and a function $\Phi\left[  \cdot\right]
:\mathcal{A}^{W_{n}}\rightarrow\mathcal{A},$ such that $(\phi(x))_{i}%
=\Phi\left[  (\sigma_{i}x)_{W_{n}}\right]  $ $.$ The \textbf{radius} of the CA
is the smallest possible $n.$
\end{theorem}

\begin{definition}
We say $X\subset\mathcal{A}^{\mathbb{Z}^{d}}$ is a \textbf{shift of finite
type} \textbf{(SFT)} if there exists $n\in\mathbb{N}$ and a finite list of
forbidden patterns $\left\{  B_{i}\right\}  \subset\mathcal{A}^{W_{n}}$ such
that $x\in X$ if and only if none of the elements of $\left\{  B_{i}\right\}
$ appear in $x.$
\end{definition}

\begin{example}
The 1D SFT obtained by the forbidden word $\left\{  11\right\}  $ (i.e. the
set of doubly infinite sequences that never have two $1^{\prime}s$ together)
is commonly known as the golden mean shift.
\end{example}

For more information on subshifts see \cite{lindmarcus}.

\bigskip

\subsection{Topological equicontinuity and local periodicity}

\begin{definition}
\label{orbitm}Given a CA $(X,\phi),$ we define the \textbf{orbit metric}
$d_{\phi}$ on $X$ as $d_{\phi}(x,y):=\sup_{i\geq0}\left\{  d(\phi^{i}%
x,\phi^{i}y)\right\}  ,$ and the \textbf{orbit balls} as
\[
O_{m}(x)\mathbf{:}=\left\{  y\mid d_{\phi}(x,y)\leq\frac{1}{2^{m}}\right\}
=\left\{  y\mid d(\phi^{i}(x),\phi^{i}(y)\leq\frac{1}{2^{m}}\text{ }\forall
i\in\mathbb{N}\right\}  .
\]

A point $x$ is an \textbf{equicontinuity point} of $\phi$ if for all
$m\in\mathbb{N}$ there exists $n\in\mathbb{N}$ such that $B_{n}(x)\subset
O_{m}(x)$. The transformation $\phi$ is \textbf{equicontinuous} if every
$x\in$ $X$ is an equicontinuity point.
\end{definition}

We have that $\phi$ is equicontinuous if and only if the family $\left\{
\phi^{i}\right\}  _{i\in\mathbb{N}}$ is equicontinuous.

\bigskip If $X$ is a subshift with dense periodic points, then a CA $(X,\phi)$
is equicontinuous if and only if it is eventually periodic (i.e. there exists
$p$ and $p^{\prime}$ such that $\phi^{p+p^{\prime}}=\phi^{p})$\cite{KurkaEQ}%
\cite{emilygamber}.

A weaker notion of periodicity that is related to equicontinuity is local periodicity.

\begin{definition}
\label{deflep}Let $(X,\phi)$ be a CA. The set of $m$-\textbf{locally periodic
}points of $\phi,$ \textbf{$LP$}$_{m}\mathbf{(}$\textbf{$\phi$}$\mathbf{)},$
is the set of points $x$ such that $(\phi^{i}x)_{W_{m}}$ is periodic (with
respect to $i$); the set of \textbf{locally periodic }points is defined as
$LP(\phi)\mathbf{:=\cap}LP_{m}(\phi)$. Similarly, the set of $m-$%
\textbf{locally eventually periodic }points of $\phi,$ \textbf{$LEP$}%
$_{m}\mathbf{(}$\textbf{$\phi$}$\mathbf{),}$ is the set of points $x$ such
that $(\phi^{i}x)_{W_{m}}$ is eventually periodic (with respect to $i$)$;$ the
set of \textbf{locally eventually periodic points }$\ $is defined as
$LEP\mathbf{(}$\textbf{$\phi$}$\mathbf{):=}\cap LEP_{m}(\phi)\mathbf{.}$

A transformation $\phi$ is \textbf{locally eventually periodic} $(LEP)$ if
$LEP(\phi)=X$ and\textbf{\ locally periodic}\emph{\ }$(LP)$ if $LP(\phi)=X.$

For $x\in LEP_{m}\mathbf{(}$\textbf{$\phi$}$\mathbf{)}$ the smallest period
will be denoted as $p_{m}(x),$ and then the smallest preperiod as $pp_{m}(x).$
\end{definition}

It is easy to see that $x\in LEP(\phi)$ if and only if $(\phi^{n}x)_{i}$ is
eventually periodic for all $i\in\mathbb{Z}.$

\begin{proposition}
\label{eqlep}Let $(X,\phi)$ be a CA. If $(X,\phi)$ is equicontinuous then it
is LEP.
\end{proposition}

\begin{proof}
Let $m\in\mathbb{N}.$ Since $X$ is compact equicontinuity implies uniform
equicontinuity, hence there exists $n\in\mathbb{N}$ such that if $y\in
B_{n}(x)$ then $y\in O_{m}(x).$

Let $x\in X$. There exist $j>j^{\prime}$such that $(\phi^{j}x)_{i}%
=(\phi^{j^{\prime}}x)_{i}$ for $\left\vert i\right\vert \leq n.$ Thus
$\phi^{j}x\in B_{n}(\phi^{j^{\prime}}x)$ and hence $\phi^{j}x\in O_{m}%
(\phi^{j^{\prime}}x).$ This implies the orbit ball is eventually periodic so
$x\in LEP_{m}(\phi);$ hence $\phi$ is $LEP$.
\end{proof}

The converse is not true.

\begin{example}
Let $X\subset\left\{  0,1\right\}  ^{\mathbb{Z}}$ be the subshift that
contains the points that contain at most one $1.$ We have that $(X,\sigma)$ is
LEP but $0^{\infty}$ is not an equicontinuity point. Hence LEP does not imply equicontinuity.
\end{example}

Nonetheless we will see that LP implies equicontinuity.

The following is an unpublished result by\ Chandgotia \cite{Nishant}. We give
the proof for completeness.

\begin{proposition}
Let $(X,\sigma)$ be a one-sided subshift. If $X$ has infinitely many periodic
points then $X$ contains a non-periodic point.
\end{proposition}

\begin{proof}
For $x\in\mathcal{A}^{\mathbb{N}}$ a periodic point$,$ the set $\mathcal{B}%
_{r}(x)$ denotes all the possible words of size $r$ that appear in $x;$ we
also define $\mathcal{B}(x):=\cup\mathcal{B}_{r}(x);$ the minimal period of
the word $w\in\mathcal{B}(x)$ in $x$ (i.e. the minimal "space" in x between
consecutive ocurrences of $w$) is denoted by $p_{x}(w).$ We say $X$ has
bounded periodic words if for all $w\in\cup_{x\text{ periodic}}\mathcal{B}(x)$
there exists $n_{w}$ such that $p_{x}(w)\leq n_{w}$ for all periodic points
$x$.

If $x\in\mathcal{A}^{\mathbb{N}}$ is periodic with minimal period $n,$ then
$\left\vert \mathcal{B}_{n}(x)\right\vert =n$ and $\left\vert \mathcal{B}%
_{r}(x)\right\vert >r$ for $1\leq r<n.$

Suppose $X$ is not a subshift with bounded periodic words. So there exists
$w\in\mathcal{B}(x)$ $\ $and $x^{n}\in X$ a sequence of periodic points such
that $p_{x^{n}}(w)\geq n$ and $x^{n}$ begins with $w.$ Any limit point of the
sequence $x^{n}$ contains $w$ only once, hence it is not periodic.

Now suppose $X$ has bounded periodic words. Let $x^{1,n}$ be a sequence of
periodic points such that $p(x^{1,n})>$ $n$ and $\mathcal{B}_{1}(x^{1,n})$ is
constant. Inductively, define $x^{i,n}$ to be a subsequence of $x^{i-1,n}$
such that $\mathcal{B}_{i}(x^{i,n})$ is constant. We can then find a sequence
of points $\left(  y^{n}\right)  \in X$ such that $p(y^{n})>n$ (and hence
$\lim_{n\rightarrow\infty}\left\vert \mathcal{B}_{r}(y^{n})\right\vert >r$ for
all $r)$ and the sequence of sets $\left(  \mathcal{B}_{r}(y^{n})\right)  $ is
eventually constant for all $r$. Let $y$ be a subsequential limit of $y^{n}$
and $w\in\lim_{n\rightarrow\infty}\mathcal{B}_{r}(y^{n}).$ There exists
$n_{w}$ such that $p_{y^{n}}(w)\leq n_{w}$ for all $n,$ so $w\in
\mathcal{B}_{r}(y^{n}).$ This means that $\left\vert \mathcal{B}%
_{r}(y)\right\vert \geq\lim_{n\rightarrow\infty}\left\vert \mathcal{B}%
_{r}(y^{n})\right\vert >r$ for all $r$ and hence $y$ is not periodic. \ 
\end{proof}

\begin{corollary}
Let $X$ be a one-sided subshift that contains only $\sigma-$periodic\ points.
Then $X$ is finite and hence $(X,\sigma)$ is equicontinuous.
\end{corollary}

\begin{proposition}
Let $(X,\phi)$ be a CA. If $(X,\phi)$ is LP then it is equicontinuous.
\end{proposition}

\begin{proof}
Let $X_{j}=\left\{  y\in\mathcal{A}^{\mathbb{N}}\mid y_{i}=(\phi^{i}%
x)_{j}\text{ for some }x\in X\right\}  ,$ i.e. the space of all sequences that
appear as the $j-th$ column of a spacetime diagram$.$ We have that
$(X_{j},\sigma)$ is a one-sided subshift that contains only periodic points,
hence there are only finitely many. This means $\phi$ is equicontinuous.
\end{proof}

\bigskip A point $x$ is \textbf{recurrent} if for every open neighbourhood $U$
the orbit of $x$ under $\phi$ intersects $U$ infinitely often; the set of
recurrent points is denoted by $R(\phi).$ The following lemma will be useful later.

\begin{lemma}
\label{nw}Let $(X,\phi)$ be a CA. Then $R(\phi)\cap LEP(\phi)=LP(\phi).$
\end{lemma}

\begin{proof}
Using the definitions it is easy to see that $LP(\phi)\subset R(\phi)\cap
LEP(\phi).$

Let $x\in R(\phi)\cap LEP(\phi)$, $m\in\mathbb{N},$ and $q:=pp_{m}(x)$ (see
Definition \ref{deflep})$.$ Suppose $x\notin LP(\phi).$ This means that $q>0,$
$(\phi^{i+q}x)_{W_{m}}$ is periodic for $i\geq0$, and $(\phi^{q-1}x)_{W_{m}%
}\neq(\phi^{q+p_{m}(x)-1}x)_{W_{m}}.$ Using the continuity of $\phi,$ we know
there exists $m^{\prime}\geq m$ such that for every $y\in B_{m^{\prime}}%
(\phi^{q-1}x),$ $(\phi^{i+q}x)_{W_{m}}=(\phi^{i+q}y)_{W_{m}}$ for $0\leq i\leq
p_{m}(x).$ Using the fact that $\phi^{q-1}x$ is recurrent we obtain that there
exists $N>pp_{m}(x)+p_{m}(x)$ such that $\phi^{N}x\in B_{m^{\prime}}%
(\phi^{q-1}x).$ This means that $(\phi^{N}x)_{W_{m}}\neq(\phi^{q+p_{m}%
(x)-1}x)_{W_{m}}$and $(\phi^{N+i}x)_{W_{m}}=(\phi^{q+p_{m}(x)+i-1}x)_{W_{m}}$
for $0<i\leq p_{m}(x).$ This is a contradiction since the second condition and
the fact that $p_{m}(x)$ is the smallest period implies that there exists
$j>0$ such that $N=q+jp_{m}(x)-1,$ and hence$(\phi^{N}x)_{W_{m}}%
=(\phi^{q+p_{m}(x)-1}x)_{W_{m}}.$
\end{proof}

\subsection{Measure theoretical equicontinuity and local periodicity}

We will use $\mu$ to denote Borel probability measures on $X.$ These do not
need to be invariant under $\phi$. 

\begin{definition}
Let $(X,\phi)$ be a CA and $\mu$ a Borel probability measure on $X$. A point
$x\in X$ is a $\mu-$\textbf{equicontinuity }point\textbf{\ }of $\phi$ if for
all $m\in\mathbb{N},$ one has%
\[
\lim_{n\rightarrow\infty}\frac{\mu(B_{n}(x)\cap O_{m}(x))}{\mu(B_{n}(x))}=1.
\]
$(X,\phi)$ is $\mu-$\textbf{equicontinuous} if almost every $x\in X$ is a
$\mu-$equicontinuity\textbf{ }point. In this case we also say $\mu$ is $\phi
-$\textbf{equicontinuous}.
\end{definition}

The concept of $\mu-$equicontinuity first appeared in \cite{Gilman1}
\cite{Gilman2}, and it was used to classify cellular automata using Bernoulli
measures. There exist CAs that have no equicontinuity points and that are
$\mu-$equicontinuous for every ergodic Markov chain (see Example
\ref{product1}).

Cellular automata with $\mu-$equicontinuous directional dynamics (e.g. when
$\sigma\circ\phi$ is $\mu-$equicontinuous) were studied in \cite{Sablik20081}.

The following result is a consequence of Corollary 7 and Theorem 9 in
\cite{mueqtds}.

\begin{theorem}
[ \cite{mueqtds}]\label{3i1}Let $(X,\phi)$ be a CA and $\mu$ a Borel
probability measure. The following are equivalent:

1) $(X,\phi)$ is $\mu-$equicontinuous

2) For every $\varepsilon>0$ there exists a compact set $M$ such that
$\mu(M)>1-\varepsilon$ and $\left.  \phi\right\vert _{M}$ is equicontinuous.

3) There exists $X^{\prime}\subset X$ such that $X^{\prime}$ is $d_{\phi}%
-$separable and $\mu(X^{\prime})=1.$
\end{theorem}

\begin{definition}
Let $(X,\phi)$ be a CA. If $\mu(LEP(\phi))=1$, we say $(X,\phi)$ is $\mu
-$\textbf{locally eventually periodic} ($\mu-LEP)$ and $\mu$ is\textbf{\ }%
$\phi-$\textbf{locally eventually periodic}\emph{\ (}$\phi-LEP).$ We define
$\mu-LP$ and $\phi-LP$ analogously.
\end{definition}

The concept of $\mu-LEP$ is new in the literature nonetheless it was motivated
by Proposition 5.2 in \cite{Gilman1}.

\begin{definition}
\label{ym}Let $m\in\mathbb{N}$ and $\varepsilon>0.$ We define
\[
Y_{\varepsilon}^{m}:=\left\{  x\mid x\in LEP(\phi),\text{ with }p_{m}(x)\leq
p_{\varepsilon}^{m}\text{ and }pp_{m}(x)\leq pp_{\varepsilon}^{m}\right\}  .
\]

\end{definition}

\begin{remark}
On the set $Y_{\varepsilon}^{m}$ we are only considering a finite possibility
of preperiods and periods. This implies that $Y_{\varepsilon}^{m}$ is equal to
a finite union of orbit balls of size $m;$ hence it is Borel.
\end{remark}

\begin{lemma}
\label{definitiony}Let $m\in\mathbb{N}$ and $\varepsilon>0.$ If $(X,\phi)$ is
$\mu-LEP$ then there exist positive integers $p_{\varepsilon}^{m}$ and
$pp_{\varepsilon}^{m}$ such that $\mu(Y_{\varepsilon}^{m})>1-\varepsilon.$
\end{lemma}

\begin{proof}
Let
\[
Y:=\cup_{s,k\in\mathbb{N}}\left\{  x\mid x\in LEP(\phi),\text{ with }%
p_{m}(x)\leq s\text{ and }pp_{m}(x)\leq k\right\}  .
\]

Since $(X,\phi)$ is $\mu-LEP$ we have that $\mu(Y)=1.$ Monotonicity of the
measure gives the desired result.
\end{proof}

Given a $\mu-LEP$ transformation $\phi$ and $m$, $\varepsilon>0,$ we will use
$p_{\varepsilon}^{m}$ and $pp_{\varepsilon}^{m}$ to denote a particular choice
of integers that satisfy the conditions of the previous lemma and that satisfy
that $p_{\varepsilon}^{m}\rightarrow\infty$ and $pp_{\varepsilon}%
^{m}\rightarrow\infty,$ as $\varepsilon\rightarrow0.$

Given a subshift $X$, we denote the $\sigma-$periodic points by $P_{X}%
(\sigma).$

The following result was proved in \cite{Gilman2} when $X$ is a full shift,
but the same result holds when $X$ is an SFT.

\begin{lemma}
\label{gilman3}Let $X$ be a 1D SFT with forbidden words of size $q$,
$(X,\phi)$ a CA with radius $r$. If there is a point $x$ and an integer
$m\neq0$ such that $O_{i}(x)\cap\sigma^{-m}O_{i}(x)\neq\emptyset$ with $i\geq
q,r$ then $O_{i}(x)\cap P_{X}(\sigma)\neq\emptyset$.
\end{lemma}

\bigskip The following proposition is proven in greater generality in
\cite{mueqtds}.

\begin{proposition}
[\cite{mueqtds}]\label{lepeq}Let $(X,\phi)$ be a CA. If $(X,\phi)$ is
$\mu-LEP$ then it is $\mu-$equicontinuous.
\end{proposition}

The converse of this proposition is not true in general (see counter-example
in \cite{mueqtds}). We obtain the converse with an extra hypothesis$.$

\begin{proposition}
\label{1i3}Let $X$ be a 1D SFT, $(X,\phi)$ a CA$,$ and $\mu$ a $\sigma
-$invariant probability measure on $X$. Then $(X,\phi)$ is $\mu-$%
equicontinuous if and only if it is $\mu-LEP.$
\end{proposition}

\begin{proof}
Let $X$ be a 1D SFT with forbidden words of size $q$, $(X,\phi)$ a CA with
radius $r$ and $p\geq q,r$. If $x$ is a $\mu-$equicontinuous point then%
\[
\lim_{n\rightarrow\infty}\frac{\mu(B_{n}(x)\cap O_{p}(x))}{\mu(B_{n}(x))}=1.
\]
Using $\mu(O_{p}(x))>0$ and Poincare's recurrence theorem we obtain that
\[
\left\{  y\mid\sigma^{i}(y)\in O_{p}(x)\text{ }i.o.\right\}
\]
is not empty$.$ Using $p\geq q,r$ and Lemma \ref{gilman3} we conclude that
every orbit ball with positive measure contains a $\sigma-$periodic point and
hence $(\phi^{n}x)_{j}$ is eventually periodic for $-p\leq j\leq p$. The
reverse implication is obtained with Proposition \ref{lepeq}.
\end{proof}

Proposition \ref{1i3} shows that $\mu-$equicontinuity and $\mu-LEP$ are
equivalent if $X$ is a 1D SFT and $\mu$ a $\sigma-$invariant measure. We do
not know if this result holds for cellular automata on multidimensional SFTs.
We can show a weaker result (Proposition \ref{aqui}) by strengthening the
$\mu-$equicontinuity hypothesis.

Let $(X,\rho)$ be a metric space$,$ and $A\subset X$. The closure of $A$ is
denoted with $cl_{\rho}(A).$

Recall that $d_{\phi}$ denotes the orbit metric (Definition \ref{orbitm}).

\begin{lemma}
\label{periodic}Let $(X,\phi)$ be a CA. If $\mu(cl_{d_{\phi}}(P_{X}%
(\sigma)))=1,$ then $(X,\phi)$ is $\mu-$LEP.
\end{lemma}

\begin{proof}
Using the fact that the $\phi-$image of a $\sigma-$periodic point is $\sigma
-$periodic with at most the same period one can see that any point in
$P_{X}(\sigma)$ is eventually periodic for $\phi.$ Let $O_{m}$ be an orbit
ball of size $m.$ This means that if $O_{m}\cap P_{X}(\sigma)\neq\emptyset$
and $x\in O_{m}$ then $x\in LEP_{m}(\phi).$ Hence if $\mu(cl_{d_{\phi}}%
(P_{X}(\sigma)))=1$ then $\phi$ is $\mu-$LEP.
\end{proof}

We represent the change of metric identity map by $\Gamma:(X,d)\rightarrow
(X,d_{\phi})$. A point $x\in X$ is an equicontinuity point of $\phi$ if and
only if it is a continuity point of $\Gamma.$ Hence $\phi$ is equicontinuous
if and only if $\Gamma$ is continuous.

Let $(X,\phi)$ be a CA. We denote the set of equicontinuity points with
$EQ(\phi).$

In \cite{BlanchardEQ} $\sigma-$ergodic measures that give full measure to the
equicontinuity points of a CA were studied. As a consequence of Lemma 3.1 (in
that paper) we can see that if $X$ is a 1D subshift, $(X,\phi)$ a CA$,$ and
$\mu$ a $\sigma-$ergodic probability measure on $X$ with $\mu(EQ(\phi))=1$
then $(X,\phi)$ is $\mu-LEP.$

If we assume the subshift has dense periodic points (or more generally are
dense in a set of full measure) then we obtain that result for
multidimensional subshifts.

\begin{proposition}
\label{aqui}Let $(X,\phi)$ be a CA and $\mu$ a measure such that $\mu
(cl_{d}(P_{X}(\sigma)))=1$. If $\mu(EQ(\phi))=1$ then $(X,\phi)$ is $\mu-LEP.$
\end{proposition}

\begin{proof}
Since equicontinuity points have full measure, then $\Gamma$ is continuous on
a set of full measure. This means that for almost every $x\in cl_{d}%
(P_{X}(\sigma)),$ we have $x\in cl_{d_{\phi}}(P_{X}(\sigma)).$ So
$\mu(cl_{d_{\phi}}(P_{X}(\sigma)))=\mu(cl_{d}(P_{X}(\sigma)))=1.$ Using Lemma
\ref{periodic} we conclude that $(X,\phi)$ is $\mu-LEP.$
\end{proof}

\bigskip We already noted that if $X$ is a subshift with dense $\sigma
-$periodic points, and $(X,\phi)$ is an equicontinuous CA then $(X,\phi)$ is
eventually periodic. Proposition \ref{aqui} is a measure theoretic analogous result.

When $X$ is an SFT then this result is also a consequence of Proposition
\ref{1i3}.

Now we present some examples.

\begin{example}
Let $x$ be a non $\sigma-$periodic point and let $\mu$ be the delta measure
supported on $\left\{  x\right\}  .$ Then $\phi=\sigma$ is $\mu-$%
equicontinuous but not $\mu$-$LEP.$
\end{example}

\begin{definition}
A 1D CA $(X,\phi)$ has \textbf{right radius 0} if there exist $L\in\mathbb{N}$
and a function $\phi\left[  \cdot\right]  :\mathcal{A}^{L}\rightarrow
\mathcal{A},$ such that $(\phi(x))_{i}=\phi\left[  x_{-L}...x_{i-1}%
x_{i}\right]  $
\end{definition}

\begin{example}
[\cite{Gilman1}]\label{slide}Let $X=\left\{  -1,0,1\right\}  ^{\mathbb{Z}},$
and $(X,\phi)$ a radius 1 CA with right radius 0 defined as follows:
$\phi\left[  11\right]  =1,\phi\left[  10\right]  =1,\phi\left[  1-1\right]
=0,\phi\left[  a1\right]  =0$ if $a\neq1,$ $\phi\left[  ab\right]  =b$ if
$a,b\neq1.$ The reader can picture the $-1s$ and $0s$ as not moving
($\phi\left[  ab\right]  =b$ if $a,b\neq1$) and the $1s$ moving to the right
at speed one until they encounters a $-1$ and the position converts into a $0$
($\phi\left[  1-1\right]  =0$)$.$ It is easy to see that this CA does not
contain equicontinuity points. For Bernoulli measures this CA is $\mu
-$equicontinuous when $\mu(-1)>\mu(1)$ \cite{Gilman1}$.$
\end{example}

We define the set $E_{i}:=\left\{  x\in\left\{  0,1\right\}  ^{\mathbb{Z}}\mid
x_{j}=1\text{ for }0\leq j\leq i\right\}  .$

\begin{example}
\label{product1}There exists a CA on the full 2-shift with no equicontinuity
points that is $\mu-$equicontinuous for every $\sigma-$invariant measure that
satisfies $\sum_{i\geq0}\mu(E_{i})<\infty$ ( in particular every non-trivial
ergodic Markov chain).
\end{example}

\begin{proof}
On $\left\{  0,1\right\}  ^{\mathbb{Z}}$ we define $\phi(x)=y$ as
$y_{i}=x_{i-1}x_{i-2}.$

One can check that for every $i>0,$ $(\phi^{i}x)_{0}=$ $\prod\limits_{j=-i}%
^{-2i}x_{j}.$ Let $a\in\left\{  0,1\right\}  .$ If there exists $n>0$ such
that $x_{i}=a$ for $i\leq-n,$ then $(\phi^{i}x)_{0}=a$ for $i\geq2n.$ This
means that for every ball $B$ there exists $x,y\in B$ such that $O_{0}(x)\neq
O_{0}(y),$ so the CA has no equicontinuity points. Let $\overline{E}%
_{i}:=\left\{  x\mid x_{j}=1\text{ for }-2i\leq j\leq-i\right\}  .$ We have
that $\sum_{i\geq0}\mu(\overline{E}_{i})=\sum_{i\geq0}\mu(E_{i})<\infty$. By
the Borel-Cantelli Lemma we have that $\mu(\overline{E}_{i}$ infinitely
often$)=0.$ This means that the probability that $(\phi^{i}x)_{0}$ has
infinitely many ones is zero; since the same argument can be given for
$(\phi^{i}x)_{m}$ we conclude $\phi$ is $\mu-LEP$ and hence $\mu
-$equicontinuous$.$

For every non-trivial ergodic Markov chain $\mu(E_{i})$ decreases
exponentially so $\sum_{i\geq0}\mu(E_{i})<\infty.$
\end{proof}

Note that in the trivial case (i.e. when $\mu(1)=1$ or $0),$ the hypothesis is
not satisfied but we also conclude $\phi$ is $\mu-$equicontinuous.

Q: Does there exists a CA with no equicontinuity points that is $\mu
-$equicontinuous for every $\sigma-$invariant $\mu?$

For more examples of $\mu-$equicontinuous CA see \cite{TisseurEQ}.

\bigskip

The following diagrams illustrate how the different properties relate on the
topological and measure theoretical level.

\textbf{Topological}%
\[
LP\Rightarrow Equicontinuous\Rightarrow LEP
\]

\textbf{Measure theoretical}%

\[
\mu-LP\Rightarrow\mu-LEP\Rightarrow\mu-equicontinuous
\]

If $X$ is a 1D SFT and $\mu$ $\sigma-$invariant then
\[
\mu-LEP=\mu-equicontinuous
\]

\section{Limit behaviour}

\subsection{Weak convergence}

A sequence of measures $\mu_{n}$ (on $X$) \textbf{converges weakly}%
\textit{\ to }$\mu_{\infty}$ $($ denoted as $\mu_{n}\rightarrow^{w}\mu
_{\infty})$ if for every continuous function $f:X\rightarrow\mathbb{R},$%
\[
\int fd\mu_{n}\rightarrow\int fd\mu_{\infty}\text{.}%
\]

This form of convergence is called weak convergence in the Probability
literature and weak* convergence in the Functional Analysis literature.

One can study limit behaviour of dynamical systems by studying the long term
behaviour of $\phi^{n}\mu$ ($\phi\mu$ is the push-forward of the measure) or
of its Cesaro averages: $\mu_{n}^{c}:=\frac{1}{n}\sum_{i=1}^{n}\phi^{i}(\mu).$
In particular we may ask if $\phi^{n}\mu$ or $\mu_{n}^{c}$ converges weakly,
and which are the properties of the limit measure.

\begin{theorem}
[Portmanteau \cite{billingsley2009convergence} pg. 15]\bigskip We have
$\mu_{n}\rightarrow^{w}\mu_{\infty}$ if and only if for every open set $U,$
$\mu_{\infty}(U)\leq\lim\inf\mu_{n}(U)$ if and only if $\mu_{n}(E)\rightarrow
\mu_{\infty}(E)$ for every set $E$ with zero boundary measure$.$
\end{theorem}

We will see that orbit balls form a weak convergence determining class when
the limit measure is $\phi-$ equicontinuous (Lemma \ref{weakcolumn}). Note
that even when $\mu$ is $\phi-LEP,$ the measure of the boundary of an orbit
ball is not necessarily zero. For example, one can check that the orbit balls
of Example\ \ref{slide} are each contained in their own boundary.

\begin{lemma}
\label{weakcolumn}Let $(X,\phi)$ be a CA, $\mu_{n}$ be a sequence of measures,
and $\mu_{\infty}$ a $\phi-$equicontinuous measure. If for every orbit ball
$A$ we have that $\mu_{n}(A)\rightarrow\mu_{\infty}(A)$ then $\mu
_{n}\rightarrow^{w}\mu_{\infty}$. Also, if $\mu$ and $\mu^{\prime}$ are
$\phi-$equicontinuous and $\mu(A)=\mu^{\prime}(A)$ for every orbit ball $A$
then $\mu=\mu^{\prime}.$
\end{lemma}

\begin{proof}
If we have two orbit balls $O$ and $O^{\prime},$ then either $O\cap O^{\prime
}=\emptyset$ or one is contained in the other. This implies we have
convergence for finite unions of orbit balls (since they can be written as
unions of disjoint orbit balls). Let $U$ be an open set. We have that $U$ is
the countable union of balls. From Theorem \ref{3i1} we know there exists a
$d_{\phi}-$separable set $X^{\prime}$ such that $\mu_{\infty}(X^{\prime})=1$ .
This means that for every $\delta>0$ there exist a finite number of orbit
balls $O_{i}$ such that $\cup_{i=1}^{N}O_{i}\subset U$ and $\mu_{\infty
}(U)\leq\mu_{\infty}(\cup_{i=1}^{N}O_{i})+\delta.$ We have that
\[
\mu_{\infty}(U)-\delta\leq\mu_{\infty}(\cup_{i=1}^{N}O_{i})=\lim\mu_{n}%
(\cup_{i=1}^{N}O_{i})\leq\lim\inf\mu_{n}(U),
\]
and hence%
\[
\mu_{\infty}(U)\leq\lim\inf\mu_{n}(U).
\]

Therefore $\mu_{n}\rightarrow^{w}\mu_{\infty}.$

Now suppose $\mu(A)=\mu^{\prime}(A)$ for every orbit ball $A$. Let $X^{\prime
}$ be the $d_{\phi}-$separable set with $\mu(X^{\prime})=1$ This means
$\mu(X^{\prime})=\mu^{\prime}(X^{\prime})=1$ (that is because $X^{\prime}$ is
equal to the union of countably many orbit balls, i.e. the balls of the
topology of $d_{\phi}$). Thus $\mu$ and $\mu^{\prime}$ agree on a $\Gamma
-$system (family closed under finite intersections) that generate the Borel
sigma algebra (intersected with $X^{\prime}$)$;$ we conclude $\mu=\mu^{\prime
}.$
\end{proof}

\begin{definition}
For $x\in LEP_{m}(\phi)$ we define
\[
O_{m}^{-q}(x):=\left\{  y\mid\exists i\in\mathbb{N}\text{ s.t. }\phi
^{ip_{m}(x)+q}y\in O_{m}(x)\right\}  .
\]

\end{definition}

We denote the Cesaro average of $\phi^{i}\mu$ with $\mu_{n}^{c},$ i.e.
$\mu_{n}^{c}:=\frac{1}{n}\sum_{i=1}^{n}\phi^{i}(\mu).$

In the following proposition we show Cesaro convergence holds for orbit balls.
Note in the proof that the convergence is actually stronger than Cesaro; there
is convergence along periodic subsequences.

In some of the followin proofs we will use $Y_{\varepsilon}^{m}$; see
Definition \ref{ym}.

\begin{lemma}
\label{limlep}Let $(X,\phi)$ be a $\mu-LEP$ CA, and $m\in\mathbb{N}$. If $x\in
LP_{m}(\phi)$ then
\[
\mu_{n}^{c}(O_{m}(x))\rightarrow\frac{1}{p_{m}(x)}\sum_{q=0}^{p_{m}(x)-1}%
\mu(O_{m}^{-q}(x)).
\]

Furthermore if $x\notin LP_{m}(\phi)$ then $\mu_{n}^{c}(O_{m}(x))\rightarrow
0.$
\end{lemma}

\begin{proof}
We have
\[
\cup_{n\geq0}\phi^{-p_{m}(x)n-q}(O_{m}(x))=O_{m}^{-q}(x).
\]
For every $0\leq q<p_{m}(x)$ and $n\geq0$
\[
\phi^{-p_{m}(x)n-q}(O_{m}(x))\subset\phi^{-p_{m}(x)(n+1)-q}(O_{m}(x)).
\]
This implies $\mu(\phi^{-p_{m}(x)n-q}(O_{m}(x)))$ is non-decreasing and
\begin{equation}
\lim_{n\rightarrow\infty}\mu(\phi^{-p_{m}(x)n-q}(O_{m}(x)))=\mu(O_{m}%
^{-q}(x)). \label{form}%
\end{equation}
Since we have convergence along periodic subsequences we have that
\[
\lim_{n\rightarrow\infty}\mu_{n}^{c}(O_{m}(x))=\frac{1}{p_{m}(x)}\sum
_{q=0}^{p_{m}(x)-1}\mu(O_{m}^{-q}(x)).
\]

Let $\varepsilon>0$ and $x\notin LP_{m}(\phi)$. If $np\prime>pp_{\varepsilon
}^{m},$ then%
\[
\phi^{-p^{\prime}n-s}(O_{m}(x))\cap Y_{\varepsilon}^{m}=\emptyset,
\]

so $\mu(\phi^{-p^{\prime}n-s}(O_{m(x)}(x)))<\varepsilon.$
\end{proof}

\begin{proposition}
\label{invlp}Let $(X,\phi)$ be a $\mu-LEP$ CA. If $\phi\mu=\mu$ then
$(X,\phi)$ is $\mu-LP.$
\end{proposition}

\begin{proof}
Using the invariance of $\mu$ and Poincare's recurrence theorem we obtain that
the set of recurrent points has full measure, i.e. $\mu(R(\phi))=1.$ By Lemma
\ref{nw} we have that
\[
\mu(LP(\phi))=\mu(LEP(\phi)\cap R(\phi))=1.
\]

\end{proof}

\begin{remark}
\label{limeas}Let $B$ be a finite union of balls (thus $B$ is compact). If
$B=\cup_{i=1}^{\infty}B_{i},$ where $\left\{  B_{i}\right\}  $ is a disjoint
family of balls, then there exists $K$ such that $B=\cup_{i=1}^{K}B_{i}.$ From
this fact we get that any premeasure on the algebra generated by the balls can
be extended to a measure on the Borel sigma-algebra.
\end{remark}

The existence of a limit measure in the following result is a natural
generalization of a result for 1D CA in \cite{BlanchardEQ} ($X$ is
multidimensional and $\phi$ may not have any equicontinuity points). Here we
also show that the limit measure is $\phi-LP.$

\begin{theorem}
\label{cesaro}\bigskip Let $(X,\phi)$ be a $\mu-LEP$ CA. The sequence of
measures $\mu_{n}^{c}$ converges weakly to a $\phi-LP$ measure $\mu_{\infty}$.
\end{theorem}

\begin{proof}
Let $B_{m}$ be a ball. For every $\varepsilon>0$ and $n\in\mathbb{N}$ we have
that
\[
\left\vert \frac{1}{n}\sum_{i=1}^{n}\mu(\phi^{-i}(B_{m}\cap Y_{\varepsilon
}^{m}))-\frac{1}{n}\sum_{i=1}^{n}\mu(\phi^{-i}(B_{m}))\right\vert \leq
n\varepsilon/n=\varepsilon.
\]
Consequently
\[
\lim_{\varepsilon\rightarrow0}\frac{1}{n}\sum_{i=1}^{n}\mu(\phi^{-i}(B_{m}\cap
Y_{\varepsilon}^{m}))=\frac{1}{n}\sum_{i=1}^{n}\mu(\phi^{-i}(B_{m}))
\]
uniformly on $n.$

On the other hand for every $\varepsilon>0$ there exists a finite set of
disjoint orbit balls $\left\{  O_{m_{k}}(x_{k})\right\}  $ such that the
$x_{k}$ are $LEP$ and $B_{m}\cap Y_{\varepsilon}^{m}=\cup_{k=1}^{k=K}O_{m_{k}%
}(x_{k})$. This implies that \
\begin{align*}
&  \frac{1}{n}\sum_{i=1}^{n}\mu(\phi^{-i}(B_{m}\cap Y_{\varepsilon}^{m}))\\
&  =\frac{1}{n}\sum_{i=1}^{n}\mu(\phi^{-i}(\cup_{k=1}^{k=K}O_{m_{k}}%
(x_{k})))\text{.}%
\end{align*}

By Lemma \ref{limlep} $\lim_{n\rightarrow\infty}\frac{1}{n}\sum_{i=1}^{n}%
\mu(\phi^{-i}(\cup_{k=1}^{k=K}O_{m_{k}}(x_{k})))$ exists. \ 

Let
\[
F(n,\varepsilon):=\frac{1}{n}\sum_{i=1}^{n}\mu(\phi^{-i}(B_{m}\cap
Y_{\varepsilon}^{m})).
\]
We have shown that
\begin{align*}
&  \lim_{n\rightarrow\infty}F(n,\varepsilon)\text{ exists, and }\\
\lim_{\varepsilon\rightarrow0}F(n,\varepsilon)  &  =\frac{1}{n}\sum_{i=1}%
^{n}\mu(\phi^{-i}(B_{m}))\text{ uniformly on }n\text{.}%
\end{align*}

Thus we obtain that $\lim_{n\rightarrow\infty}\lim_{\varepsilon\rightarrow0}$
$F(n,\varepsilon)$ exists, and%
\begin{align*}
\lim_{n\rightarrow\infty}\lim_{\varepsilon\rightarrow0}\frac{1}{n}\sum
_{i=1}^{n}\mu(\phi^{-i}(B_{m}\cap Y_{\varepsilon}^{m}))  &  =\lim
_{n\rightarrow\infty}\frac{1}{n}\sum_{i=1}^{n}\mu(\phi^{-i}(B_{m}))\\
&  =\lim_{n\rightarrow\infty}\mu_{n}^{c}(B_{m}).
\end{align*}

The proof of the previous statement is common in analysis (see for example
Theorem 1 in \cite{kadelburg2005interchanging}).

We define $\mu_{\infty}$ as the measure that satisfies $\mu_{\infty}%
(B_{m}):=\lim_{n\rightarrow\infty}\mu_{n}^{c}(B_{m})$ (see Remark
\ref{limeas}) for every ball $B_{m}$. Every open set $U$ can be approximated
by a finite disjoint union of balls$.$ This implies $\mu_{\infty}(U)\leq
\lim\inf\mu_{n}^{c}(U),$ and hence $\mu_{n}^{c}\rightarrow^{w}\mu_{\infty}$.

Since $\phi^{-1}(LEP(\phi))=LEP(\phi)$ we have that $\mu_{n}^{c}$ is
$\phi-LEP.$ For every $m\in\mathbb{N}$ and $\varepsilon>0$ we have that
$Y_{\varepsilon}^{m}$ is a finite union of orbit balls and hence.
\[
\mu_{\infty}(Y_{\varepsilon}^{m})=\lim\mu_{n}^{c}(Y_{\varepsilon}^{m}%
)\geq1-\varepsilon.
\]
Since $Y_{\varepsilon}^{m}\subset LEP(\phi)$ we obtain that $\mu_{\infty}$ is
$\phi-LEP.$ Considering that $\phi\mu_{\infty}=\mu_{\infty}$ and Proposition
\ref{invlp} we obtain that $(X,\phi)$ is $\mu-LP.$
\end{proof}

There is a more general definition of $\mu-LEP$ for topological dynamical
systems (see \cite{mueqtds}). It is possible to check that the previous result
holds for topological dynamical systems on zero dimensional spaces.

\subsection{Behaviour of $\mu_{\infty}$}

In this section we prove the main results of the paper. In this section we
will use $\mu_{\infty}$ to denote the weak limit of $\mu_{n}^{c}.$

\begin{proposition}
\label{aprox}Let $(X,\phi)$ be a $\mu-LEP$ CA. Then $\phi^{n}\mu
\rightarrow^{w}\mu_{\infty}$ if and only if $\phi^{n}\mu(O)\rightarrow
\mu_{\infty}(O)$ for all orbit balls $O,$ and $\mu_{n}^{c}$ $\rightarrow
^{w}\mu_{\infty}$ if and only if $\mu_{n}^{c}(O)\rightarrow\mu_{\infty}(O)$
for all orbit balls $O.$
\end{proposition}

\begin{proof}
Assume that $\phi^{n}\mu\rightarrow^{w}\mu_{\infty}.$

Let $m\in\mathbb{N}$ and $x\in LEP(\phi).$ Take $\varepsilon>0$ so that
$p_{\varepsilon}^{m}>$ $p_{m}(x).$

Let $O_{m}(x)=\cap_{i\in\mathbb{N}}\phi^{-i}G_{i},$ where every $G_{i}$ is a
ball defined by the $ith$ row of the spacetime diagram of $O_{m}(x)$. There
exists $k_{1}$ such that
\[
\left\vert \mu_{\infty}(\cap_{i=1}^{k}\phi^{-i}G_{i})-\mu_{\infty}%
(O_{m}(x))\right\vert \leq\varepsilon\text{ for }k\geq k_{1}.
\]

Fix $k\geq2p_{\varepsilon}^{m},k_{1}.$ If $n\geq pp_{\varepsilon}^{m}$
$\ $then
\[
\phi^{-n}(\cap_{i=1}^{k}\phi^{-i}G_{i})\cap Y_{\varepsilon}^{m}=\phi
^{-n}(O_{m}(x))\cap Y_{\varepsilon}^{m}.
\]
Since $\mu(Y_{\varepsilon}^{m})>1-\varepsilon$ we obtain%
\[
\left\vert \phi^{n}\mu(\cap_{i=1}^{k}G_{i})-\phi^{n}\mu(O_{m}(x))\right\vert
\leq\varepsilon\text{ for }n\geq pp_{\varepsilon}^{m}.
\]

Since $\cap_{i=1}^{k}\phi^{-i}G_{i}$ has no boundary, there exists $N$ such
that%
\[
\left\vert \phi^{n}\mu(\cap_{i=1}^{k}\phi^{-i}G_{i})-\mu_{\infty}(\cap
_{i=1}^{k}\phi^{-i}G_{i})\right\vert \leq\varepsilon\text{ for }n\geq N.
\]

Using the inequalities we obtain%
\[
\left\vert \phi^{n}\mu(O_{m}(x))-\mu_{\infty}(O_{m}(x))\right\vert
\leq3\varepsilon\text{ for }n\geq N,pp_{\varepsilon}^{m}.
\]

Hence $\phi^{n}\mu(O_{m}(x))\rightarrow\mu_{\infty}(O_{m}(x)).$

The other direction is a corollary of Lemma \ref{weakcolumn}.

The proof for $\mu_{n}^{c}$ is analogous.
\end{proof}

\begin{definition}
For $a\in\mathbb{R}$ and $E\subset X$ Borel$,$ we define
\[
A_{a}^{E}:=\left\{  y:\lim_{n\rightarrow\infty}\frac{1}{\left\vert
W_{n}\right\vert }\sum_{i\in W_{n}}1_{E}(\sigma^{i}(y))=a\right\}  ,
\]
where $W_{n}=\left[  -n,n\right]  ^{d}.$
\end{definition}

Note that if $\mu$ is $\sigma-$ergodic then by the pointwise ergodic theorem
\[
\mu(A_{a}^{E})=\left\{
\begin{array}
[c]{cc}%
1 & \text{if }a=\mu(E)\\
0 & otherwise
\end{array}
\right.  .
\]

\begin{lemma}
\label{avrg}Let $\mu$ be a $\sigma-$ergodic measure, $(X,\phi)$ a $\mu-LEP$
CA$,$ $m\in\mathbb{N}$, and $x\in LP_{m}(\phi).$ If for every $0\leq
q<p_{m}(x)$ there exists $N_{q}\in\mathbb{N}$ such that%
\[
\mu(\phi^{-p_{m}(x)n-q}(O_{m}(x)))=\mu(O_{m}^{-q}(x))\text{ for all }n\geq
N_{q}%
\]
then
\[
\mu_{n}^{c}(A_{a}^{O_{m}(x)})\rightarrow\mu_{\infty}(A_{a}^{O_{m}(x)}%
)=\frac{1}{p_{m}(x)}\sum_{q=0}^{p_{m}(x)-1}\mu(A_{a}^{O_{m}^{-q}(x)}).
\]

\end{lemma}

\begin{proof}
By hypothesis we have that for $n\geq N_{q}$
\[
\phi^{p_{m}(x)n+q}\mu(O_{m}(x))=\mu(O_{m}^{-q}(x)).
\]
Note that since $\mu$ is $\sigma-$ergodic then $\phi^{n}\mu$ is $\sigma
-$ergodic for every $n\geq1.$ Let $0\leq q<p_{m}(x).$

If $n\geq N_{q}$ then%
\begin{align*}
\phi^{p_{m}(x)n+q}\mu(A_{a}^{O_{m}(x)})  &  =\left\{
\begin{array}
[c]{cc}%
1 & \text{if }a=\phi^{p_{m}(x)n+q}\mu(O_{m}(x))\\
0 & otherwise
\end{array}
\right. \\
&  =\left\{
\begin{array}
[c]{cc}%
1 & \text{if }a=\mu(O_{m}^{-q}(x))\\
0 & \text{ \ \ \ \ \ \ \ \ \ \ \ }otherwise
\end{array}
\right. \\
&  =\mu(A_{a}^{O_{m}^{-q}(x)}).
\end{align*}

This implies we have convergence along periodic subsequences, thus%
\[
\mu_{n}^{c}(A_{a}^{O_{m}(x)})\rightarrow\mu_{\infty}(A_{a}^{O_{m}(x)}%
)=\frac{1}{p_{m}(x)}\sum_{q=0}^{p_{m}(x)-1}\mu(A_{a}^{O_{m}^{-q}(x)}).
\]

\end{proof}

Using the fact that $\mu(\phi^{-p_{m}(x)n-q}(O_{m}(x)))\rightarrow\mu
(O_{m}^{-q}(x))$ one can check that one of the hypotheses of this result is
always satisfied if we assume the CA is $\mu-LP.$

\begin{lemma}
\label{remarco}\bigskip\ Let $(X,\phi)$ be a $\mu-LP$ CA$,$ $m\in\mathbb{N}$,
and $x\in LP_{m}(\phi).$ Then for every $0\leq q<p_{m}(x)$ we have%
\[
\mu(\phi^{-p_{m}(x)i-q}(O_{m}(x)))=\mu(O_{m}^{-q}(x))\text{ for all }%
i\in\mathbb{N}.
\]

\end{lemma}

\begin{proof}
This comes from the fact that $LP(\phi)\cap\phi^{-p_{m}(x)i-q}(O_{m}%
(x))=LP(\phi)\cap O_{m}^{-q}(x)$ for all $i\in\mathbb{N}.$
\end{proof}

\begin{theorem}
\label{mme copy(1)}Let $\mu$ be a $\sigma-$ergodic measure and $(X,\phi)$ a
$\mu-LP$ CA. Then $\mu_{\infty}$ is $\sigma-$ergodic if and only if $\mu$ is
$\phi-$invariant$.$
\end{theorem}

\begin{proof}
If $\phi\mu=\mu$ then $\mu=\mu_{\infty}$ and hence it is $\sigma-$ergodic.

Suppose $\mu_{\infty}$ is $\sigma-$ergodic and $\mu$ is not $\phi-$invariant.
By Lemma \ref{weakcolumn} we know there exists $x\in$ $LP(\phi)$ and
$m\in\mathbb{N}$ such that
\[
\mu(O_{m}(x))\neq\mu(\phi^{-1}O_{m}(x)).
\]

This implies that that $p_{m}(x)\geq2.$

We have that
\begin{align*}
O_{m}(x)\cap LP(\phi)  &  =O_{m}^{0}(x)\cap LP(\phi),\text{ and}\\
\phi^{-1}O_{m}(x)\cap LP(\phi)  &  =O_{m}^{-1}(x)\cap LP(\phi).
\end{align*}
Since $(X,\phi)$ is $\mu-LP$ we obtain that
\[
\mu(O_{m}^{0}(x))\neq\mu(O_{m}^{-1}(x)).
\]

Using Lemma \ref{avrg} and Lemma \ref{remarco} we get
\[
\mu_{\infty}(A_{a}^{O_{m}(x)})=\frac{1}{p_{m}(x)}\sum_{q=0}^{p_{m}(x)-1}%
\mu(A_{a}^{O_{m}^{-q}(x)}).
\]
We reach a contradiction because $A_{\mu(O_{m}^{0}(x))}^{O_{m}(x)}$ is
$\sigma-$invariant, $\mu_{\infty}$ is $\sigma-$ergodic but
\[
\frac{1}{p_{m}(x)}\leq\mu_{\infty}(A_{\mu(O_{m}^{0}(x))}^{O_{m}(x)})\leq
\frac{p_{m}(x)-1}{p_{m}(x)}.
\]

\end{proof}

Every subshift has at least one \textbf{measure of maximal entropy (MME)},
i.e. a measure whose entropy is the same as the topological entropy of the
subshift. A 1D subshift is \textbf{irreducible} if for every pair of balls
$U,V$ there exists $j\in\mathbb{Z}$ such that $\sigma^{j}U\cap V\neq
\emptyset.$ A well known result of Shannon and Parry states that every
irreducible 1D SFT admits a unique MME, and it always has full support
\cite{parry1964intrinsic}. The MME of a fullshift is the uniform Bernoulli measure.

We note that we are only discussing measures of maximal entropy with respect
to the shift not to $\phi.$

\begin{theorem}
[Coven-Paul \cite{covenpaul}]\label{coven}Let $X$ be a 1D irreducible SFT with
a unique MME, and $(X,\phi)$ a CA. Then $(X,\phi)$ is surjective if and only
if it preserves the MME. In particular if $\phi:\mathcal{A}^{\mathbb{Z}%
}\rightarrow\mathcal{A}^{\mathbb{Z}}$ is a CA, then $(X,\phi)$ is surjective
if and only if it preserves the uniform Bernoulli measure.
\end{theorem}

\begin{lemma}
\label{nonw}\bigskip Let $(X,\phi)$ be CA. Assume that $\phi$ preserves a
measure with full support. If $x$ is an equicontinuity point then $x$ is
recurrent $(x\in R(\phi))$.
\end{lemma}

\begin{proof}
Let $x$ be an equicontinuity point and $m\in\mathbb{N}$. There exists
$n\geq2m$ such that $B_{n}(x)\subset O_{2m}(x).$ We have that $\phi$ preserves
a fully supported measure. Using Poincare's recurrence theorem we conclude
there exists $j\in\mathbb{N}$ such that $\phi^{j}B_{n}(x)\cap B_{n}%
(x)\neq\emptyset.$ This implies that $\phi^{j}x\in B_{m}(x)$, and thus $x\in
R(\phi).$
\end{proof}

\bigskip

In \cite{BlanchardEQ} it was asked under which conditions the limit measure,
under a 1D CA, of $\sigma-$ergodic measures that give full measure to
equicontinuity points i.e. $\mu(EQ(\phi))=1,$\ is $\sigma-$ergodic, a measure
of maximal entropy or $\phi-$ergodic$.$ We address those questions.

\begin{theorem}
\label{sigmaergodic}Let $X$ be a 1D irreducible SFT, $(X,\phi)$ a surjective
CA, and $\mu$ a $\sigma-$ergodic measure with $\mu(EQ(\phi))=1$. Then
$\mu_{\infty}$ is $\sigma-$ergodic if and only if $\mu$ is $\phi-$invariant.
\end{theorem}

\begin{proof}
If $\mu=\phi\mu,$ then $\mu_{\infty}=\mu$ is $\sigma-$ergodic$.$

Since $\phi$ is surjective by Shannon-Parry and Coven-Paul we obtain that
$\phi$ preserves a fully supported measure. By Proposition \ref{aqui} we have
that $\mu(LEP(\phi))=1.$ Using Lemma \ref{nonw} and Lemma \ref{nw} we get that
$\mu(LP(\phi))=\mu(LEP(\phi)\cap R(\phi))=1;$ and hence $(X,\phi)$ is
$\mu-LP.$ Using Theorem \ref{mme copy(1)} we obtain $\mu=\phi\mu.$
\end{proof}

The proof of the following result is similar.

\begin{theorem}
\label{mmecor}Let $X$ be a 1D irreducible SFT, $(X,\phi)$ a CA, and $\mu$ a
$\sigma-$ergodic measure with $\mu(EQ(\phi))=1$. Then $\mu_{\infty}$ is the
MME if and only if $\mu$ is the MME and $(X,\phi)$ is surjective.
\end{theorem}

\begin{proof}
If $\mu$ is the MME and $\mu=\phi\mu,$ then $\mu_{\infty}=\mu$ is the MME$.$

Assume $\mu_{\infty}$ is the MME; hence it is $\sigma-$ergodic and has full
support. By Proposition \ref{aqui} we have that $\mu(LEP(\phi))=1.$ Using
Lemma \ref{nonw} and Lemma \ref{nw} we get that $\mu(LP(\phi))=\mu
(LEP(\phi)\cap R(\phi))=1;$ and hence $(X,\phi)$ is $\mu-LP.$ Using Theorem
\ref{mme copy(1)} we obtain $\mu=\phi\mu=\mu_{\infty}.$
\end{proof}

\bigskip\ There is interest in dynamical systems such that the orbit of the
measure will converge in some sense to a measure of maximal entropy or an
equilibrium markov measure (for example see Section 4.4 of \cite{stat}). It
has been shown that the markov measures under 1D linear permutive CA converge
in Cesaro sense to the measure of maximal entropy (e.g. \cite{LindXOR}%
\cite{pivatoalgebraic}\cite{MassMartinezCesaro}). Theorem \ref{mmecor} shows
that if a measure $\mu$ is not the measure of maximal entropy and the
equicontinuity points have full measure then the limit measure will not be the
measure of maximal entropy.

Under some conditions these results hold for multidimensional subshifts.

One of the implications of the mentioned theorem by Coven-Paul has been
generalized. Let $X$ be a subshift with a unique MME, and $(X,\phi)$ a CA. If
$(X,\phi)$ is surjective then it preserves the MME (Theorem 3.3 in
\cite{Meester01higher-dimensionalsubshifts}). The proof of the following
results are almost the same as the proofs for Theorems \ref{mmecor} and
\ref{sigmaergodic}.

\begin{theorem}
Let $X$ be an subshift with dense periodic points and a unique and fully
supported MME, $(X,\phi)$ a surjective CA, and $\mu$ a $\sigma-$ergodic
measure with $\mu(EQ(\phi))=1$. Then $\mu_{\infty}$ is $\sigma-$ergodic if and
only if $\mu$ is $\phi-$invariant.
\end{theorem}

\begin{theorem}
Let $X$ be an subshift with dense periodic points and a unique and fully
supported MME, $(X,\phi)$ a CA, and $\mu$ a $\sigma-$ergodic measure with
$\mu(EQ(\phi))=1$. Then $\mu_{\infty}$ is the MME if and only if $\mu$ is the
MME and $\mu=\phi\mu.$
\end{theorem}

For $\mu-LEP$ systems we can show sufficient conditions for the $\sigma
-$ergodicity of $\mu_{\infty}$.

\begin{lemma}
Let $\mu$ be a $\phi-LEP$ measure$.$Then $\mu_{\infty}$ is $\sigma-$ergodic if
and only if for every $x\in LP(\phi).$
\[
\mu_{\infty}(A_{a}^{O_{m}(x)})=\left\{
\begin{array}
[c]{cc}%
1 & if\text{ }a=\mu_{\infty}(O_{m}(x))\\
0 & \text{otherwise}%
\end{array}
\right.  .
\]

\end{lemma}

\begin{proof}
The $\Rightarrow)$ implication is given by the pointwise ergodic theorem.

If the equation is satisfied then the pointwise ergodic theorem conclusion
holds for all sets of the form $O_{m}(x)$ with $x\in LP_{m}(\phi)$. By
Proposition \ref{invlp} $\mu_{\infty}(LP(\phi))=1.$ Since $\left\{
O_{m}(x)\mid m\in\mathbb{N}\text{ and }x\in LP_{m}(\phi)\right\}  $ generates
the Borel sigma algebra (intersected with $LP(\phi)$)$,$ we conclude
$\mu_{\infty}$ is $\sigma-$ergodic (see \cite{walters2000introduction}\ pg.41.)
\end{proof}

\begin{theorem}
Let $\mu$ be a $\phi-LEP$, $\sigma-$ergodic measure$.$ If for every orbit ball
$O,$ with $\mu_{\infty}(O)>0,$ there exists $N_{O}$ such that
\[
\phi^{n}\mu(O)=\mu_{\infty}(O)\text{ \ for }n\geq N_{O},
\]

then $\mu_{\infty}$ is $\sigma-$ergodic.
\end{theorem}

\begin{proof}
Let $m\in\mathbb{N},$ $x\in LP_{m}(\phi)$ and $0\leq q<p_{m}(x).$ From the
proof of equation (\ref{form}) of Lemma \ref{limlep} one can see that for
every $0\leq q<p_{m}(x)$, $\phi^{p_{m}(x)n+q}\mu(O_{m}(x))$ is non-decreasing
and converges $($as $n\rightarrow\infty)$. Using this and the hypothesis we
have that there exists $N$ such that
\[
\phi^{p_{m}n+q}\mu(O_{m}(x))=\mu_{\infty}(O_{m}(x))\text{ for }n\geq N.
\]
This implies
\[
\mu(O_{m}^{-q}(x))=\mu_{\infty}(O_{m}(x)).
\]
Using Lemma \ref{avrg} we obtain
\[
\mu_{\infty}(A_{a}^{O_{m}(x)})=\frac{1}{p_{m}(x)}\sum_{r=0}^{p_{m}(x)-1}%
\mu(A_{a}^{O_{m}^{-r}(x)})\text{ for every }a\in\mathbb{R}\text{.}%
\]
This implies
\[
\mu_{\infty}(A_{a}^{O_{m}(x)})=\mu(A_{a}^{O_{m}^{-q}(x)})\text{ for every
}a\in\mathbb{R}.
\]

Using the $\sigma-$ergodicity of $\mu$ we get
\[
\mu(A_{a}^{O_{m}^{-q}(x)})=\left\{
\begin{array}
[c]{cc}%
1 & if\text{ }a=\mu(O_{m}^{-q}(x))\\
0 & \text{otherwise}%
\end{array}
\right.  .
\]
Hence, we obtain
\[
\mu_{\infty}(A_{a}^{O_{m}(x)})=\left\{
\begin{array}
[c]{cc}%
1 & if\text{ }a=\mu_{\infty}(O_{m}(x))\\
0 & \text{otherwise}%
\end{array}
\right.  .
\]

Using the previous lemma we conclude that $\mu_{\infty}$ is $\sigma-$ergodic.
\end{proof}

\bigskip

\bigskip We will now study measure preserving dynamical systems. We say
$(M,T,\mu)$ is a measure preserving transformation if $(M,\mu)$ is a measure
space, $T:M\rightarrow M$ is measurable and $T\mu=\mu.$ When we say $\mu$ is
ergodic we also assume it is invariant under $T$.

Two measure preserving transformations $(M_{1},T_{1},\mu_{1})$ and
$(M_{2},T_{2},\mu_{2})$ are\textbf{ isomorphic (measurably) } if there exists
an invertible measure preserving transformation\textit{\ }$f:(X_{1},\mu
_{1})\rightarrow(X_{2},\mu_{2}),$ such that the inverse is measure preserving
and $T_{2}\circ f=f\circ T_{1}.$

The spectral theory for dynamical systems (TDS and measure preserving
transformations) is useful for studying rigid transformations. We will give
the definitions and state the most important results. For more details and
proofs see \cite{walters2000introduction}.

A measure preserving transformation $T$ on a measure space $(M,\mu)$ generates
a unitary linear operator on the Hilbert space $L^{2}(M,\mu),$ by
$U_{T}:f\mapsto f\circ T,$ known as the \textbf{Koopman operator}$.$ The
spectrum of the Koopman operator is called the \textbf{spectrum}\textit{\ }of
the measure preserving transformation. The spectrum is \textbf{pure point
}or\textbf{\ discrete} if there exists an orthonormal basis for $L^{2}(M,\mu)$
which consists of eigenfunctions of the Koopman operator$.$ The spectrum
is\textit{\ }\textbf{rational} if the eigenvalues are complex roots of unity.
Classical results by Halmos and Von Neumann state that two ergodic measure
preserving transformation with discrete spectrum have the same group of
eigenvalues if and only if they are isomorphic, and that an ergodic measure
preserving transformation has pure point spectrum if and only if it is
isomorphic to a rotation on a compact metric group. The eigenfunctions of a
rotation on a compact group are generated by the characters of the group.
Discrete spectrum can be characterized for topological dynamical systems using
a weak forms of $\mu-$equicontinuity \cite{weakeq}.

\begin{example}
\label{odom}Let $S=(s_{0},s_{1},...)$ be a finite or infinite sequence of
integers larger or equal than 1. The $S-$\textbf{adic odometer} is the
$+(1,0,...)$ (with carrying) map defined on the compact set $D=\prod
\nolimits_{i\geq0}\mathbb{Z}_{s_{i}}$ (for a survey on odometers see
\cite{surveyodometers}).
\end{example}

These transformations are also called adding machines. An ergodic measure
preserving transformation has discrete rational spectrum if and only if it is
isomorphic to an odometer.

Any odometer can be embedded in a CA \cite{yassawiodometer}.

The following result was proved in \cite{mueqtds}.

\begin{proposition}
Let $(X,\phi)$ be a CA and $\mu$ a $\phi-$invariant measure. If $(X,\phi)$ is
$\mu-LEP$ then $(X,\phi,\mu)$ has discrete rational spectrum.
\end{proposition}

This implies that every ergodic $\mu-LEP$ CA is isomorphic to an odometer.

We obtain a stronger result if the measure is $\sigma-$invariant.

\begin{proposition}
\label{mulp}Let $(X,\phi)$ be a $\mu-$ergodic, $\mu-LEP$ CA. If $\mu$ is
$\sigma-$invariant then $(X,\phi,\mu)$ is isomorphic to a cyclic permutation
on a finite set.
\end{proposition}

\begin{proof}
By Proposition \ref{invlp} $\mu(LP(\phi))=1.$ Since $\mu-LP$ CA are $\mu
-$equicontinuous, there exists $x\in LP(\phi)$ such that $\mu(O_{0}(x))>0.$
Let $O_{0}^{\infty}$ be the $\phi-$orbit of $O_{0}(x).$ We have that
$\mu(O_{0}^{\infty})>0.$ Since $\phi(O_{0}^{\infty})=O_{0}^{\infty}$ and $\mu$
is $\phi-$ergodic, then $\mu(O_{0}^{\infty})=1.$

We have that $p_{0}(O_{0}^{\infty})=\left\{  p_{0}(x)\right\}  .$ This implies
the $0th$ column of almost every point is periodic with period $p_{0}(x).$
Since $\mu$ is $\sigma-$invariant we have that almost every point is $\phi
-$periodic (with period $p_{0}(x)$).
\end{proof}

Using this and other assumptions we can characterize when a limit measure of a
$\mu-$equicontinuous CA is $\phi-$ergodic.

\begin{corollary}
\label{odometer}Let $X$ be a 1D SFT, $\mu$ be a $\sigma-$invariant measure,
$(X,\phi)$ be a $\mu-$equicontinuous CA. We have that $(X,\phi,\mu_{\infty})$
is isomorphic to a cyclic permutation on a finite set if and only if
$\mu_{\infty}$ is $\phi-$ergodic.
\end{corollary}

\bigskip

The following corollary combines the main results of this paper.

\begin{corollary}
Let $X$ be a 1D irreducible SFT, $(X,\phi)$ CA$,$ and $\mu$ a shift-ergodic
$\mu-$equicontinuous measure. Let $\mu_{\infty}$ be the weak limit of the
Cesaro average of $\phi^{n}\mu.$ We have that

$\cdot\mu_{\infty}$ is the measure of maximal entropy if and only if $\mu$ is
the measure of maximal entropy and $\phi\mu=\mu=\mu_{\infty}.$

$\cdot\mu_{\infty}$ is $\phi-$ergodic if and only if $(X,\mu_{\infty},\phi)$
is isomorphic (measurably) to a cyclic permutation on a finite set.

Furthermore if we assume $\phi$ is surjective then

$\cdot\mu_{\infty\text{ }}$is $\sigma-$ergodic if and only if $\mu$ is $\phi-$invariant.
\end{corollary}

\bibliographystyle{aabbrv}
\bibliography{camel}

\newif\ifabfull\abfullfalse
\input apreambl
\begin{thebibliography}{10}

\bibitem{AkinAuslander}
\abtype{0}{E.~Akin, J.~Auslander\abphrase{1}\abphrase{0}K.~Berg}.
\newblock When is a transitive map chaotic?
\newblock \abtype{2}{Ohio State Univ. Math. Res. Inst. Publ.}
  \abtype{3}{5}\abtype{4}{2}, 25--40 \abtype{5}{1996}.

\bibitem{billingsley2009convergence}
\abtype{0}{P.~Billingsley}.
\newblock \abtype{1}{Convergence of Probability Measures}.
\newblock Wiley Series in Probability and Statistics. Wiley \abtype{5}{2009}.

\bibitem{BlanchardEQ}
\abtype{0}{F.~Blanchard\abphrase{0}P.~Tisseur}.
\newblock Some properties of cellular automata with equicontinuity points.
\newblock \abtype{2}{Annales de l'Institut Henri Poincare (B) Probability and
  Statistics} \abtype{3}{36}\abtype{4}{5}, 569 -- 582 \abtype{5}{2000}.

\bibitem{Nishant}
\abtype{0}{N.~Chandgotia}.
\newblock personal communication.
\newblock \abtype{5}{2013}.

\bibitem{covenpaul}
\abtype{0}{E.~Coven\abphrase{0}M.~Paul}.
\newblock Endomorphisms of irreducible subshifts of finite type.
\newblock \abtype{2}{Mathematical systems theory} \abtype{3}{8}\abtype{4}{2},
  167--175 \abtype{5}{1974}.

\bibitem{yassawiodometer}
\abtype{0}{E.~Coven\abphrase{0}R.~Yassawi}.
\newblock Embedding odometers in cellular automata.
\newblock \abtype{2}{preprint} \abtype{5}{2009}.

\bibitem{surveyodometers}
\abtype{0}{T.~Downarowicz}.
\newblock Survey of odometers and toeplitz flows.
\newblock \abtype{2}{Contemporary Mathematics} \abtype{3}{385}, 7--38
  \abtype{5}{2005}.

\bibitem{emilygamber}
\abtype{0}{E.~Gamber}.
\newblock Equicontinuity properties of d-dimensional cellular automata.
\newblock \abtype{2}{Topology Proceedings} \abtype{3}{30}\abtype{4}{1}
  \abtype{5}{2006}.

\bibitem{mueqtds}
\abtype{0}{F.~Garc{\'\i}a-Ramos}.
\newblock A characterization of $\mu-$equicontinuity for topological dynamical
  systems.
\newblock \abtype{2}{arXiv:1309.0467 [math.DS]} \abtype{5}{2015}.

\bibitem{weakeq}
\abtype{0}{F.~Garc{\'\i}a-Ramos}.
\newblock Weak forms of topological and measure theoretical equicontinuity:
  relationships with discrete spectrum and sequence entropy.
\newblock \abtype{2}{Ergodic Theory and Dynamical Systems (to appear)
  arXiv:1402.7327 [math.DS]} \abtype{5}{2015}.

\bibitem{Gilman2}
\abtype{0}{R.~Gilman}.
\newblock Periodic behavior of linear automata.
\newblock \abphrase{7}\abtype{0}{J.~Alexander}\abphrase{4},
  \abtype{1}{Dynamical Systems}, \abphrase{8}
  1342\abphrase{5}\abtype{1}{Lecture Notes in Mathematics}, \abphrase{13}
  216--219. Springer Berlin / Heidelberg \abtype{5}{1988}.
\newblock 10.1007/BFb0082833.

\bibitem{Gilman1}
\abtype{0}{R.~H. Gilman}.
\newblock {Classes of linear automata.}
\newblock \abtype{2}{Ergodic Theory Dyn. Syst.} \abtype{3}{7}, 105--118
  \abtype{5}{1987}.

\bibitem{hedlund}
\abtype{0}{G.~Hedlund}.
\newblock Endomorphisms and automorphisms of the shift dynamical systems.
\newblock \abtype{2}{Mathematical System Theory} \abtype{3}{3}, 320 -- 375
  \abtype{5}{1969}.

\bibitem{kadelburg2005interchanging}
\abtype{0}{Z.~Kadelburg\abphrase{0}M.~M. Marjanovic}.
\newblock Interchanging two limits.
\newblock \abtype{2}{Enseign. Math} \abtype{3}{8}, 15--29 \abtype{5}{2005}.

\bibitem{stat}
\abtype{0}{J.~Kari\abphrase{0}S.~Taati}.
\newblock Statistical mechanics of surjective cellular automata.
\newblock \abtype{2}{preprint arXiv:1311.2319} \abtype{5}{2014}.

\bibitem{KurkaEQ}
\abtype{0}{P.~Kurka}.
\newblock Languages, equicontinuity and attractors in cellular automata.
\newblock \abtype{2}{Ergodic Theory and Dynamical Systems}
  \abtype{3}{17}\abtype{4}{02}, 417--433 \abtype{5}{1997}.

\bibitem{LindXOR}
\abtype{0}{D.~Lind}.
\newblock {Applications of ergodic theory and sofic systems to cellular
  automata.}
\newblock {Cellular automata, Proc. Interdisc. Workshop, Los Alamos/N.M. 1983,
  36-44 (1984).} \abtype{5}{1984}.

\bibitem{lindmarcus}
\abtype{0}{D.~Lind\abphrase{0}B.~Marcus}.
\newblock \abtype{1}{An Introduction to Symbolic Dynamics and Coding}.
\newblock Cambridge University Press \abtype{5}{1995}.

\bibitem{MassMartinezCesaro}
\abtype{0}{A.~Maass\abphrase{0}S.~Martinez}.
\newblock {On Ces\`aro limit distribution of a class of permutative cellular
  automata.}
\newblock \abtype{2}{J. Stat. Phys.} \abtype{3}{90}\abtype{4}{1-2}, 435--452
  \abtype{5}{1998}.

\bibitem{Meester01higher-dimensionalsubshifts}
\abtype{0}{R.~Meester\abphrase{0}J.~E. Steif}.
\newblock Higher-dimensional subshifts of finite type, factor maps and measures
  of maximal entropy.
\newblock \abtype{2}{Pac. J. Math} \abphrase{13} 497--510 \abtype{5}{2001}.

\bibitem{parry1964intrinsic}
\abtype{0}{W.~Parry}.
\newblock Intrinsic markov chains.
\newblock \abtype{2}{Transactions of the American Mathematical Society}
  \abphrase{13} 55--66 \abtype{5}{1964}.

\bibitem{pivatoalgebraic}
\abtype{0}{M.~Pivato\abphrase{0}R.~Yassawi}.
\newblock Limit measures for affine cellular automata.
\newblock \abtype{2}{Ergodic Theory and Dynamical Systems}
  \abtype{3}{22}\abtype{4}{04}, 1269--1287 \abtype{5}{2002}.

\bibitem{Sablik20081}
\abtype{0}{M.~Sablik}.
\newblock Directional dynamics for cellular automata: A sensitivity to initial
  condition approach.
\newblock \abtype{2}{Theoretical Computer Science} \abtype{3}{400}, 1 -- 18
  \abtype{5}{2008}.

\bibitem{TisseurEQ}
\abtype{0}{P.~Tisseur}.
\newblock Density of periodic points, invariant measures and almost
  equicontinuous points of cellular automata.
\newblock \abtype{2}{Advances in Applied Mathematics}
  \abtype{3}{42}\abtype{4}{4}, 504 -- 518 \abtype{5}{2009}.

\bibitem{walters2000introduction}
\abtype{0}{P.~Walters}.
\newblock \abtype{1}{An Introduction to Ergodic Theory}.
\newblock Graduate Texts in Mathematics. Springer-Verlag \abtype{5}{2000}.

\end{thebibliography}

\end{document}